\documentclass[11pt,a4paper]{amsart}

\usepackage{amsmath, amsthm, amssymb, graphicx, color, euscript, appendix,enumitem,microtype,hyperref,mathtools, parskip}

\newtheorem{theorem}{Theorem}
\newtheorem{lemma}{Lemma}[section]
\newtheorem{corollary}[lemma]{Corollary}
\newtheorem*{question}{Question}
\newtheorem*{lifting lemma}{Lifting Lemma}
\newtheorem*{splitting lemma}{Splitting Lemma}
\newtheorem*{acknowledgment}{Acknowledgment}
\newtheorem*{Theorem 1 with parameters}{Theorem 1 with parameters}
\newtheorem*{Theorem 2 with parameters}{Theorem 2 with parameters}

\theoremstyle{definition}
\newtheorem{definition}[lemma]{Definition}

\theoremstyle{remark}
\newtheorem{remark}[lemma]{Remark}
\newtheorem{example}[lemma]{Example}

\newcommand{\C}{\mathbb{C}} 
 
\newcommand{\N}{\mathbb{N}}

\newcommand{\B}{\mathbb{B}}

\newcommand{\cC}{{\ensuremath{\mathcal{C}}}}
\newcommand{\cE}{{\ensuremath{\mathcal{E}}}}
\newcommand{\cF}{{\ensuremath{\mathcal{F}}}}
\newcommand{\cG}{{\ensuremath{\mathcal{G}}}}

\newcommand{\cO}{{\ensuremath{\mathcal{O}}}}

\newcommand{\Ueul}{\EuScript{U}}

\DeclareMathOperator{\Id}{Id}
\DeclareMathOperator{\pr}{pr}

\begin{document}
\title{A splitting lemma for coherent sheaves}
\author{Luca Studer}
\email{luca.studer@math.unibe.ch}
\begin{abstract}
The presented splitting lemma extends the techniques of Gromov and Forstneri\v c 
to glue local sections of a given analytic sheaf, a key step in the proof of all Oka principles. 
The novelty on which the proof depends is a lifting lemma for transition maps of coherent sheaves, 
which yields a reduction of the proof to the work of Forstneri\v c. As applications we get shortcuts in the 
proofs of Forster and Ramspott's Oka principle for admissible pairs and of the interpolation property of sections of elliptic submersions, an extension of Gromov's results obtained by 
Forstneri\v c and Prezelj. 
\end{abstract} 

\maketitle

\section{introduction}
Splitting lemmata are often the crux in the proof of an Oka principle. It was Gromov's idea to split transition maps~\cite{elliptic bundles}, 
which gave rise to proofs in Oka theory that were unthinkable before the late eighties. 
Gromov's approach was refined and generalized mainly by Forstneri\v c, see e.g.~\cite{Runge implies Oka, Francs book}. 
The purpose of this note is to establish Theorem~\ref{splitting}, a splitting lemma which extends 
Forsteneri\v c's results to transition maps of coherent sheaves. The proof of Theorem~\ref{splitting} is a consequence of 
Forstneri\v c's splitting lemma, see Proposition 5.8.4 in~\cite{Francs book}, and Theorem~\ref{lifting} of this note, a lifting lemma 
for transition maps of coherent sheaves. Applications of the obtained splitting lemma are discussed in the 
last section along Forster and Ramspott's Oka principle for admissible pairs~\cite{Forster und Ramspott} and 
 the interpolation property of the sheaf of holomorphic sections of an elliptic submersion~\cite{FP}. An appendix is 
 devoted to parametric versions of Theorem~\ref{lifting} and~\ref{splitting}.

\subsection{Motivation} An Oka principle is a theorem stating that there are 
only topological obstructions to a solution of a given complex analytic problem. In the tradition of Oka theory such 
theorems are formulated on a  Stein space $X$, e.g. $$X=\{z \in \C^l: f_1(z)=f_2(z)=\cdots=f_m(z)=0\}$$ for holomorphic 
functions $f_1, f_2, \ldots, f_m:\C^l \to \C$. 
The proof of an Oka principle follows many times a more or less well-known strategy. 
The main work is the inductive gluing of local solutions defined on subsets of $X$, whose domains grow with every step of the induction. 
This inductive step depends on a gluing lemma, which states that from given local solutions defined on subsets $A$ and $B$ of $X$ we can obtain 
-- under suitable assumptions -- a single solution defined on the union $A\cup B$. The proof of a gluing lemma is usually reduced 
to the proof of a splitting lemma. Let us discuss the role of splitting lemmata along Cartan's division theorem. 

\begin{example}
\label{division theorem}
Let $g_1, g_2, \ldots, g_n:X \to \C$ be holomorphic functions. Clearly, a necessary assumption to solve 
$$x_1g_1+\cdots +x_ng_n=1$$ in the ring of holomorphic functions $X\to \C$ is that $g_1, \ldots, g_n$ do not 
vanish simultaneously. Cartan's division theorem states that this assumption is sufficient provided that $X$ is Stein.
To formulate the problem efficiently it is worth to consider 
the associated sheaf homomorphism
$$\pi: \cO^n_X \to \cO_X, \ \pi(y_1,\ldots, y_n)=y_1g_1+\cdots +y_ng_n.$$ 
We wish to find a global solution $x=(x_1,\ldots, x_n) \in \cO^n_X(X)$ to $\pi(x)=1$. To illustrate the inductive step in the proof, assume that there 
are open subsets $A,B\subset X$ and local solutions $a\in \cO^n_X(A)$ and $b \in \cO^n_X(B)$ to the problem in question.
Then $c\coloneqq a-b$ is defined on $C\coloneqq A\cap B$ and lies in the kernel of $\pi$ 
by linearity. Instead of solving the gluing problem directly, we can try to establish the existence of a 
section $\alpha$ of $\ker \pi$ with domain $A$ and a section $\beta$ of $\ker \pi$ with domain $B$, which satisfy the equation 
$$c+\alpha=\beta$$ on $C$. This is called an additive splitting of $c$ over $(A,B)$ in the sheaf $\ker \pi$, the kernel of $\pi$. If we can solve this splitting problem, 
we can define a solution $d$ on $D=A\cup B$ to our initial problem by setting 
$d|A=a+\alpha$ and $d|B=b+\beta$. The map $d$ is well-defined since $c+\alpha=\beta$ on $C$, 
and satisfies $\pi(d)=1$ since $\pi$ is linear and $\alpha$ and $\beta$ are in the kernel of $\pi$.
\end{example}

Until today the gluing problem in Oka theory is most of times reduced to a splitting problem. 
However, if the target of the maps in question are no longer linear, the splitting cannot be formulated additively. The following 
reformulation allows the desired generalization. Let
$$\gamma:C \times \C^n \to C \times \C^n, \ \gamma(x,z)=(x,c(x)+z),$$ where $c=a-b$ is as in Example~\ref{division theorem}. 
Then, solving the additive splitting problem for $c$ over $(A,B)$ in $\ker \pi$ is equivalent 
to solving the splitting problem $\gamma \circ \alpha=\beta$ over $(A,B)$ in $\ker \pi$, where $\alpha$ and $\beta$ are now understood 
as local sections of the trivial bundle $X \times \C^n \to X$ defined on $A$ and $B$ respectively. 
Such $\gamma$ is called a transition map of the sheaf $\ker \pi$. 

\begin{definition}
\label{transition map}
Let $p:E \to X$ be a vector bundle over a reduced Stein base $X$ and let $\cF$ be a subsheaf 
of the sheaf of holomorphic sections of $E$.
A \textit{transition map} of $\cF$ is a holomorphic map 
$\gamma:U \to E$, $U\subset E$ open, which 
preserves $\cF$ in the sense that if $\delta$ is a local section of $\cF$ 
with values in $U$, then $\gamma \circ \delta$ is likewise 
a local section of $\cF$.
\end{definition}

\begin{remark}
A transition map $\gamma: U \to E$ of the sheaf of holomorphic sections of $E$ is simply a holomorphic map which preserves the fibers of $p:E\to X$. 
In that case we call $\gamma$ a transition map and omit to mention the sheaf.
\end{remark}

In settings of Oka theory where the target of the considered analytic maps are not linear, 
the reduction of the gluing problem to a splitting problem is significantly 
harder than in the case of the proof of Cartan's division theorem, see Example~\ref{division theorem}. It 
involves many times the implicit function theorem and Cartan's theorem B. 

\subsection{Results} 
To state the results, let us make the following agreements. 
All considered complex spaces are second countable and all considered spaces of maps are equipped with the compact open topology.
Moreover, $\B^m_r\subset \C^m$ denotes the open ball of radius $r>0$ centered at the origin and we 
use the following formulation concerning approximation.

\begin{definition}
\label{approx}
For a multivalued map $\Gamma:S \to T$ of topological spaces 
we say that \textit{$t \in \Gamma(s)$ can be chosen to approximate $t_0 \in \Gamma(s_0)$ as $s$ approximates $s_0$} if for every neighborhood $U$ of $t_0$ there is 
a neighborhood $V$ of $s_0$ such that $U \cap \Gamma(s) \not = \emptyset$ for $s \in V$.
\end{definition}

\begin{theorem}[Lifting Lemma]
\label{lifting}
Let $p:E \to W$ be a vector bundle over a reduced Stein base, let $f_1, \ldots, f_m:W \to E$ be holomorphic sections and set 
$$f:W \times \C^m \to E, \ f(x,z_1, \ldots, z_m)=z_1f_1(x)+\cdots +z_mf_m(x).$$
Moreover let $r>0$ and let $U\subset E$ be an open neighborhood of $f(W \times \B^m_r)$. Then, for every transition map $\gamma:U \to E$ of $\cF=\cO_W f_1 + \cdots +\cO_W f_m$, 
there is a transition map $\tilde \gamma: W \times \B^m_r \to W \times \C^m$ with $f \circ \tilde \gamma=\gamma \circ f|\B^m_r \times W$. 
Moreover $\tilde \gamma$ can be chosen to approximate the identity as $\gamma$ approximates the identity. 
\end{theorem}

To state Theorem~\ref{splitting} we need the notion of a Cartan pair.

\begin{definition}
\label{C-pair}
Let $X$ be a complex space and $A,B \subset X$. The pair $(A,B)$ is a \textit{Cartan pair} if 
\begin{enumerate}
\item $A,B, A \cap B$ and $A \cup B$ are Stein compacts, and 
\item $\overline{A \setminus B} \cap \overline{B \setminus A}= \emptyset$.
\end{enumerate}
\end{definition}

\begin{theorem}[Splitting Lemma]
\label{splitting}
Let $p:E \to X$ be a vector bundle over a reduced Stein base and identify $X$ with the image of the zero section $X\to E$. Let $(A,B)$ be a Cartan pair in $X$ 
and $U\subset E$ a neighborhood of $A \cap B$, and let 
$\cF$ be a coherent subsheaf of the sheaf of holomorphic sections of $E$. 
Then there are neighborhoods $A'\supset A$ and $B'\supset B$ such that for every transition map $\gamma: U \to E$ of $\cF$ which 
is sufficiently close to the identity, there are $\alpha \in \cF(A')$ and $\beta \in \cF(B')$ such that 
$$\gamma \circ \alpha =\beta,$$ on $A'\cap B'$. 
Moreover, $\alpha$ and $\beta$ can be chosen to approximate 
the zero section as $\gamma$ approximates the identity.
\end{theorem}

\begin{remark}
\label{special case}
In the special case where $E=X\times \C^m$ and $\cF=\cO^m_X$, the sheaf of sections of $X \times \C^m \to X$, the conclusion of Theorem~\ref{splitting} follows from Proposition 5.8.4, page 238 in~\cite{Francs book}, a key tool in modern Oka theory. 
\end{remark}

Let us make two remarks about the coherent sheaf $\cF$ from Theorem~\ref{splitting}. 
These remarks are unimportant in view of the given proofs, but hopefully help the reader to understand $\cF$ in more concrete terms.

\begin{remark}
Note that Theorem~\ref{splitting} only makes a statement about sections of the given coherent sheaf $\cF$ 
defined on subsets of a small neighborhood $V \coloneqq A'\cup B'$ of $A \cup B$, which 
can be assumed to be Stein and relatively compact in $X$. By Cartan's theorem A and the fact that $V$ is relatively compact, the stalks of $\cF|V$ are generated by a finite family 
$f_1,\ldots, f_m$ of global sections of $E$. Expressed differently, after perhaps replacing $X$ by a relatively compact Stein neighborhood of $A\cup B\subset X$, we can assume that $\cF$ is generated 
by finitely many global sections $f_1,\ldots, f_m$. The converse is true as well: If $f_1,\ldots, f_m$ are global sections of the vector bundle $E$ over $X$, then 
the sheaf $\cF$ generated by $f_1, \ldots, f_m$ is trivially of finite type. Moreover, every sheaf of relations corresponding to $\cF$ is trivially a sheaf of relations of the sheaf $\cE$
of holomorphic sections of $E$ (recall that $\cF\subset \cE$). Since $\cE$ is known to be coherent and thus of relation finite type, $\cF$ is of relation finite type as well, completing the proof that 
$\cF$ is coherent. 
\end{remark}

\begin{remark}
For a relatively compact Stein neighborhood $V\subset X$ of $A \cup B$, the restricted vector bundle $E|V$ embeds into a trivial bundle 
$V\times \C^k$ for a sufficiently large $k \in \N$. This is a consequence of theorem A and B. Consequently, 
after perhaps replacing $X$ by a relatively compact Stein neighborhood $V$ of $A \cup B$, the sheaf 
$\cF$ from Theorem~\ref{splitting} can be considered a subsheaf of $\cO^k_X$, which is generated by finitely many holomorphic maps $f_1, \ldots, f_m: X \to \C^k$ by the previous remark.
\end{remark}

\begin{proof}[Proof of Theorem~\ref{splitting}] 
Let $W$ be a Stein neighborhood of $A \cap B$ which is relatively compact in $X\cap U$. 
By Cartan's theorem A there are finitely many global sections 
$f_1, \ldots, f_m \in \cF(X)$ which generate the stalk at every point of $W$, i.e. 
$\cF|W=\cO_{W}f_1+\cdots +\cO_{W}f_m$. Let 
\begin{align*}
f:X \times \C^m \to E, \ f(x,z_1,\ldots, z_m)=z_1f_1(x)+\cdots +z_mf_m(x)
\end{align*}
and $r>0$ such that $f(W \times \B^m_r) \subset U$. 
By Remark~\ref{special case} there are neighborhoods $A'\supset A$ and $B' \supset B$ such that for every transition map $\tilde \gamma: W \times \B^m_r \to W \times \C^m$ 
which is sufficiently close to the identity there are $\tilde \alpha \in \cO^m_X(A')$ and $\tilde \beta \in \cO^m_X(B')$ with $\tilde \gamma \circ \tilde \alpha = \tilde \beta$ on $A'\cap B'$. In addition 
$\tilde \alpha$ and $\tilde \beta$ can be chosen to approximate the zero sections as $\tilde \gamma$ approximates the identity.
Now, it follows from Theorem~\ref{lifting} that if $\gamma:U \to E$ is a transition map of $\cF$ which 
approximates the identity sufficiently well, then there is a lift $\tilde \gamma$ of $\gamma$ through $f$ for which we find such a splitting $(\tilde \alpha, \tilde \beta)$, and that the splitting $(\tilde \alpha, \tilde \beta)$ can 
be chosen to approximate the zero section as $\gamma$ (and therefore the chosen lift $\tilde \gamma$) approximates the identity. For $\alpha= f \circ \tilde \alpha \in \cF(A')$ and $\beta = f \circ \tilde \beta \in \cF(B')$ we get on $A'\cap B'$ 
\begin{align*}
\gamma \circ \alpha =\gamma \circ f \circ \tilde \alpha = f \circ \tilde \gamma \circ \tilde \alpha =f \circ \tilde \beta=\beta.
\end{align*}
This finishes the proof.
\end{proof}

We are optimistic that the presented results generalize in a natural way to a coherent sheaf $\cF$ which 
is not necessarily a subsheaf of some sheaf of sections of a vector bundle. However, such an extension would be 
unnecessarily complicated in most settings considered in Oka theory and we 
omit a careful study of the following 
question for now. 

\begin{question}
Is there a natural notion of a (holomorphic) transition map for arbitrary coherent sheaves 
with interesting applications in Oka theory?
\end{question}

\begin{acknowledgment}
\normalfont
I would like to thank Frank Kutzschebauch for suggesting the topic and an observation which led to the complete proof of Lemma~\ref{generators}. 
Moreover I would like to thank Franc Forstneri\v c for the invitation to Ljubljana in fall 2016 and stimulating discussions during the stay. 
Finally I thank Finnur L\'arusson and Filip Misev for helpful advice concerning the presentation.
\end{acknowledgment}

\section{Proof of Theorem~\ref{lifting}}
We need the notion of parametric sheaves defined in the following way. 

\begin{definition}
\label{parametric sheaf}
Let $p:E \to X$ be a vector bundle over a reduced complex space and let $\cF$ be a subsheaf of $\cO_X$-modules of the 
sheaf of holomorphic sections of $E$. We define $\cF_k$, $k\geq 1$ as the sheaf of $\cO_{X \times \C^k}$-modules of holomorphic maps $X \times \C^k \to E$, whose local sections  
$f$ satisfy in addition that $f(\cdot , z)$ is in $\cF$ whenever $z \in \C^k$ is fixed.
\end{definition}

\begin{remark}
In the case where $\cF=\cE$ is the sheaf of sections of $E$, $\cE_k$ is the sheaf of sections of the pull-back bundle of $E$ via the projection $\pr:X\times \C^k \to X$.
\end{remark}

\begin{remark}
From the viewpoint of sheaves of modules, an alternative and more abstract definition of $\cF_k$ is 
to set $\cF_k=\pr^\ast \cF$, the inverse image sheaf defined by $$\pr^\ast \cF=\pr^{-1} \cF \otimes_{\pr^{-1}\cO_X} \cO_{X \times \C^k}.$$
However, in the case where $\cF$ is not locally free, it is not obvious that the two definitions yield the same sheaf of $\cO_{X\times \C^k}$-modules. 
The proof that this is the case 
is hidden in our next lemma, the only ingredient in the preparation of the proof of Theorem~\ref{lifting} for which we are not aware of a good 
reference in the literature.
\end{remark}

In the following we consider $\cF$ as a subsheaf of $\cF_k$ given by those elements which do not depend on $z \in \C^k$.

\begin{lemma}
\label{generators}
Let $p:E\to X$ be a vector bundle over a reduced complex space, let $\cF$ be a subsheaf of $\cO_X$-modules of the sheaf of holomorphic sections of $E$ and let $k\geq 1$. 
If $S \subset \cF_q$ generates the stalk of $\cF$ at some $q\in X$, then $S$ generates the 
stalks of $\cF_k$ at every point $(q,z)$, $z \in \C^k$.
\end{lemma}

The given proof of Lemma~\ref{generators} depends on the following well-known results.

\begin{lemma}
\label{finitely generated}
Any submodule $M$ of a stalk of $\cO^n_{\C^l}$ is finitely generated.
\end{lemma}

\begin{proof}
This follows from the facts that the stalks of $\cO_{\C^l}$ are Noetherian (see e.g.~\cite{Gunning/Rossi}, Theorem 9, page 72) and 
that if a ring $R$ is Noetherian, then the submodules of $R^n$ are finitely generated.
\end{proof}
 
\begin{lemma}
\label{Frechet}
Let $p:E \to X$ be a vector bundle over a reduced complex space and let $\cF$ be a subsheaf of $\cO_X$-modules of the sheaf of holomorphic sections $\cE$ of $E$. 
Then $\cF$ is a Fr\'echet subsheaf of $\cE$.
\end{lemma}

\begin{proof}
For details in the case $E=X \times \C^n$ and hence $\cE=\cO^n_X$ compare Proposition 2, page 235 in~\cite{Gunning/Rossi}. The 
extension to a possibly non-trivial vector bundle $E$ is easy. 
\end{proof}

We will use the following special case of Artin's approximation theorem.

\begin{lemma}
\label{Artin's approximation theorem}
Let $F=F(x,w)$ be a germ of $\cO^{n}_{\C^l \times \C^m}$ at the origin, where 
$x=(x_1, \ldots, x_l)\in \C^l$ and $w=(w_1, \ldots, w_m)\in \C^m$. Then there is a germ $h$ of $\cO^m_{\C^l}$ at the origin which solves $$F(x,h(x))=0$$ if and only if there is a formal power series solution $\hat h$ of the same equation.
\end{lemma}

\begin{proof}
See~\cite{Artin}, p. 277.
\end{proof}

\begin{proof}[Proof of Lemma~\ref{generators}]
Note that $(\cF_k)_1=\cF_{k+1}$ for $k\in \N$, hence the case $k>1$ follows from the case $k=1$ with a simple induction.
We show the case $k=1$. Since the statement is local we may assume that $E=X\times \C^n$. 
By passing  to a local model we may assume that $X \subset \C^l$ is a closed subvariety 
of some open $U\subset \C^l$ and $0 \in X$. Moreover, it suffice to consider the case $(q,z)=(0,0) \in \C^l\times \C$. 
Let $M$ be the $\cO_{\C^l,0}$-submodule of $\cO^n_{\C^l,0}$ given by the condition $$f|X \in \cF_0 \text{   if   }f \in M.$$ Similarly, 
define $M_1$ as the $\cO_{\C^l\times \C,(0,0)}$-submodule of $\cO^n_{\C^l\times \C,(0,0)}$ given by the condition 
$f(\cdot, z)|X \in \cF_0$ if $z \in \C$ is fixed.
There are $f_1, \ldots, f_m \in M$ which generate $M$ by Lemma~\ref{finitely generated}. If we can show 
that $f_1|X, \ldots, f_m|X$ generate the stalk of $\cF_1$ at $(0,0)\in X \times \C$, 
then, since $S$ generates the stalk of $\cF$ at the origin and therefore the germs of $f_1|X, \ldots, f_m|X$ at $0$, we are done. 
To show this, it suffice to show that $f_1, \ldots, f_m$ generate $M_1$, for which we have in fact 
$$M_1=\{f \in \cO^n_{{\C^l \times \C},(0,0)}: f(\cdot,z)|X\in \cF_0\}=\{f \in \cO^n_{{\C^l \times \C},(0,0)}: f(\cdot,z)\in M\}.$$ 
To show that $M_1$ is generated by $f_1, \ldots, f_m$, let $f \in M_1$. In a neighborhood of the origin $f$ is represented by an analytic map, which can be written 
as $$f(x,z)=  \sum_{i\geq 0} z^ig_i(x)$$ for suitable $g_i \in \cO^n(\Delta^l)$ and $z \in \Delta$ for a sufficiently small disc $\Delta\subset \C$ centered at $0$. 
Let $M(\Delta^l)\subset \cO^n(\Delta^l)$ be the $\cO(\Delta^l)$-module given by 
those elements of $\cO^n(\Delta^l)$ which restrict to an element of $M$. 
We show with an induction on $i\geq 0$ that $g_i \in M(\Delta^l)$. 
Since $f \in M_1$ we have that 
$g_0=f(\cdot, 0) \in M(\Delta^l)$, which completes the beginning of the induction. 
Assume for the inductive step that $g_0, \ldots, g_i \in M(\Delta^l)$. Since 
$M(\Delta^l)$ is a $\cO(\Delta^l)$-module and thus a $\C$-vector space, we get for fixed $z \in \Delta\setminus \{0\}$ from $f(\cdot, z) \in M(\Delta^l)$
\begin{align*}
\lambda_z(x)	&=g_{i+1}(x)+zg_{i+2}(x) +z^2g_{i+3}(x)+\cdots \\
		&=\frac{f(x,z) - (g_0(x)+zg_1(x)+\cdots +z^ig_i(x))}{z^i} \in M(\Delta^l).
\end{align*}
For $z\to 0$, $\lambda_z:\Delta^l \to \C^n$ converges uniformly on compacts to $g_{i+1}$. 
Since submodules of $\cO^n(\Delta^l)$ are complete with respect to the compact open topology (see Lemma~\ref{Frechet}), we get that $g_{i+1} \in M(\Delta^l)$. This finishes the inductive step.
Now, knowing that $g_i \in M$ and that $f_1, \ldots, f_m$ generate $M$, there are function germs $g_{j,i}$ at the origin of $\C^l \times \C$, $i\geq 0$, $j =1, \ldots, m$ 
such that 
\begin{align*}
g_i=g_{1,i}f_1+ \cdots +g_{m,i}f_m.
\end{align*}
In particular, purely at the level of formal power series, we get for 
$$\hat h_j(x,z)\coloneqq\sum_{i\geq 0} z^i g_{j,i}(x), \ j=1, \ldots, m$$ the equality 
$$f=\sum_{i\geq 0} z^i g_i=\sum_{i\geq 0} z^i (g_{1,i}f_1+ \cdots +g_{m,i}f_m)=\hat h_1f_1+\cdots +\hat h_mf_m.$$ 
For $w=(w_1, \ldots, w_m) \in \C^m$, applying Artin's approximation theorem (see Lemma~\ref{Artin's approximation theorem}) to the germ 
$F(x,z,w)=f(x,z)-(w_1f_1(x)+\ldots +w_mf_m(x))$ yields function germs $h_1, \ldots, h_m$ at the origin of $\C^l \times \C$ which satisfy 
$$f=h_1f_1+\cdots +h_mf_m.$$ This is what we had to show and finishes the proof. 
\end{proof}

\begin{lemma}
\label{Theorem B}
Let $D$ be Stein and $\cF$ a coherent sheaf generated by global sections $f_1, \ldots, f_m \in \cF(D)$. 
Then $\cF(D)=\cO(D)f_1+\cdots + \cO(D)f_m$.
\end{lemma}

\begin{proof}
It suffice that if $D$ is Stein and $\cF$ is coherent, then a sheaf epimorphism $\cO^m_D \to \cF$ induces an epimorphism $\cO^m(D) \to \cF(D)$, 
an application of Cartan's theorem B.
\end{proof}

\begin{proof}[Proof of Theorem~\ref{lifting}]
For $D \coloneqq W \times \B^m_r$ our assumptions imply immediately $\gamma \circ f|D \in \cF_m(D)$, where $\cF=\cO_Wf_1+\cdots +\cO_Wf_m$. By Lemma~\ref{generators} $f_1, \ldots, f_m$ 
generate $\cF_m$ as a sheaf of $\cO_{W \times \C^m}$-modules and by Lemma~\ref{Theorem B} 
we have $\cF_m(D)=\cO(D)f_1+\cdots +\cO(D)f_m$. Expressed differently, 
$$\pi: \cO^m(D) \to \cF_m(D), \ (h_1,\ldots, h_m) \mapsto h_1f_1+\cdots +h_mf_m$$ is onto. 
In particular there is $\tilde \gamma \in \cO^m(D)$ with $\pi(\tilde \gamma)= \gamma \circ f|D$. Such $\tilde \gamma$ is a lift 
of $\gamma$ through $f$ since $\pi(\tilde \gamma)=f \circ \tilde \gamma$ by definition. 
To see that we can choose the lift $\tilde \gamma$ to approximate the identity on $D$ as $\gamma$ approximates the identity on $U$, 
note that $\cF_m(D)$ is a Fr\'echet space by Lemma~\ref{Frechet}. In particular $\pi: \cO^m(D) \to \cF_m(D)$ is a surjective linear map of Fr\'echet spaces. 
If $\gamma$ approximates the identity, then $\gamma \circ f|D$ approximates $f|D=f \circ \Id_D=\pi(\Id_D)$, hence the lift 
$\tilde \gamma$ with $\pi(\tilde \gamma)=\gamma \circ f|D$ can be chosen to approximate the identity by the open mapping theorem for surjective linear maps of Fr\'echet spaces. 
This finishes the proof.
\end{proof}

\section{Applications}
In this section we give a hint how the obtained splitting lemma (Theorem~\ref{splitting}) applies in the proof of Forster and Ramspott's Oka principle for admissible pairs 
and to extension theorems of modern Oka theory. 

\subsection{Splitting in the proof of the Oka principle for admissible pairs}
This concerns the work of Forster and Ramspott, see~\cite{Forster und Ramspott}.
We get straight to the point rather than stating the corresponding Oka principle. 
In this setting one can formulate the necessary splitting in a (complex analytic) Lie group bundle $\pi:P \to X$ with fiber $G$, 
for which we give a shorter proof than the analogous result in ~\cite{Forster und Ramspott}. 
The base $X$ is as usually assumed to be reduced and Stein.
There is a Lie algebra bundle $p:E \to X$ (which is in particular a vector bundle) 
associated to $\pi$, whose fiber is the Lie algebra $\mathfrak{g}$ of $G$. Moreover, there is an exponential map $\exp:E \to P$ (of fiber bundles over $X$), 
which has a logarithm $\log$, that is an inverse to $\exp$, defined in a neighborhood of the 
identity section of $P$. The maps $\exp$ and $\log$ are the usual exponential map and logarithm corresponding to $G$ and its Lie algebra $\mathfrak{g}$ if 
restricted to single fibers of $p$ and $\pi$ respectively.
Considered are a coherent subsheaf of Lie algebras $\cF$ of the 
sheaf of holomorphic sections of $p:E \to X$ and a subsheaf of groups $\cG$ of the  sheaf of holomorphic sections of $\pi: P \to X$, which is generated 
by $\cF$ in the following sense: if $f$ is a local section of $\cF$, then $\exp \circ f$ is a local section of $\cG$. Moreover, there 
is a neighborhood $U$ of the identity section in $P$  such that the local sections of $\cG$ with values in $U$ are mapped 
bijectively to the local sections of $\cF$ with values in $\log(U)$. The desired splitting is 

\begin{corollary}
\label{group splitting}
Let $(A,B)$ be a Cartan pair of $X$ and $W$ a neighborhood of $A\cap B$. 
Then there is a neighborhood $\Ueul \subset \cG(W)$ of the identity section and neighborhoods $A'\supset A$ and $B'\supset B$ such that for every $c \in \Ueul$, 
there are 
sections $a \in \cG(A')$ and $b \in \cG(B')$ with $ca=b$ on $A'\cap B'$. Moreover, $a$ and $b$ may be chosen 
to approximate the identity section as $c$ approximates the identity section.
\end{corollary}

\begin{proof}
For a sufficiently small neighborhood $U\subset E$ of $A\cap B$ and a sufficiently small neighborhood 
$\Ueul \subset \cG(W)$ of the identity section, the map $\gamma=\gamma_c$
$$\gamma:U \to E, \ \ \gamma(e)=\log(c\circ p(e) \cdot \exp(e))$$ 
is well-defined for every $c \in \Ueul$. It follows directly from our assumptions that $\gamma$ is a transition map for 
$\cF$ defined in a neighborhood of $A\cap B$. Since $\gamma=\gamma_c$ approximates the identity as $c$ approximates the identity section 
of $\cG$, after possibly shrinking the neighborhood $\Ueul$ suitably, Theorem~\ref{splitting} yields $A'\supset A$, $B'\supset B$ such that if $c \in \Ueul$ there 
is a solution $\alpha \in \cF(A')$, $\beta \in \cF(B')$ to the splitting problem
$$\gamma \circ \alpha =\beta$$ on $A' \cap B'$. Moreover, $\alpha$ and $\beta$ may be chosen to 
approximate the zero section as $c$ and therefore $\gamma$ approximates the identity.
By the trivial fact that $c \circ p \circ \alpha=c$ on $A'\cap B'$ we get 
$\log(c \exp \circ \alpha)= \gamma \circ \alpha=\beta$ on $A'\cap B'$. By composing with $\exp$ on both sides of the equation we get 
for $a=\exp\circ \alpha$ and $b=\exp \circ \beta$ the equation 
$c a=b$ on $A'\cap B'$, and $a,b$ approximate the identity section as $c$ approximates the identity section. 
This finishes the proof.
\end{proof}

\begin{remark}
The given proof of Corollary~\ref{group splitting} is more general than the techniques used by Forster and Ramspott 
since it is independent of the Lie algebra structure of the fibers of $E$. 
\end{remark}

\subsection{Splitting with interpolation for sections of elliptic submersions}
In this subsection we discuss briefly an alternative proof of the 
natural generalization of Cartan's extension theorem to sections of elliptic 
submersions due to Forstneri\v c and Prezelj~\cite{FP}, which depends on a special case of Theorem~\ref{splitting}. 
To this end, let us recall some notions. 
Let $h:Z \to X$ be a holomorphic submersion, let 
$\pi:E\to Z$ be a vector bundle over $Z$ and identify $Z$ with the image of the zero section $Z\to E$. 
A \textit{spray} over $h$ with domain $E$ is a map $s:E \to Z$ whose restriction to $Z$ is the identity and 
which preserves the fibers of $h$ in the sense that $h \circ s=h \circ \pi$. 
A spray $s:E \to Z$ over $h$ is called \textit{dominating} if the restriction of the differential $ds$ to 
the zero section is onto $\ker dh\subset TZ$. Finally, a submersion $h:Z \to X$ is called 
\textit{elliptic} if every point $p \in X$ has a neighborhood $U$ such that the restricted submersion $h: h^{-1}(U)\to U$ has 
a dominating spray. A beautiful instance of an Oka principle is the natural 
generalization of Cartan's extension theorem from sections of the projection $X \times \C \to X$ to sections of an arbitrary 
elliptic submersion.

\begin{theorem}[Forstneri\v c, Prezelj~\cite{FP}]
Let $h:Z\to X$ be an elliptic submersion onto a reduced second countable Stein base $X$ and let $X'\subset X$ be an analytic 
subvariety. Then a holomorphic section $s:X' \to h^{-1}(X')$ of $h|X'$ extends to a holomorphic section $X \to Z$ of $h$ if and 
only if $s$ extends to a continuous section $X \to Z$ of $h$.
\end{theorem}

As in most Oka principles, a suitable splitting lemma is one of the key points in the proof. The following  
corollary is a special case of Theorem~\ref{splitting} which extends Proposition~4.1 in~\cite{FP}. 

\begin{corollary}
\label{interpolation}
Let $(A,B)$ be a Cartan pair in a reduced Stein space $X$, let $W$ be an open neighborhood of $A\cap B$ and let $r>0$. 
Then there are neighborhoods $A'\supset A$, $B'\supset B$ such that for every transition map 
$\gamma: W \times \B^n_r \to W \times \C^n$ (of $\cO^n_X$) 
which approximates the identity sufficiently well and satisfies $\gamma|W\cap (X' \times \{0\})=id$, 
there are $\alpha \in \cO^n_X(A')$, $\beta \in \cO^n_X(B')$, both vanishing on $X'$, such that  
$$\gamma \circ \alpha =\beta$$ on $A'\cap B'$. 
Moreover, $\alpha$ and $\beta$ may be chosen to approximate 
the zero section as $\gamma$ approximates the identity.
\end{corollary}

\begin{proof}
Since $\gamma|W\cap (X' \times \{0\})$ is equal to the identity, $\gamma$ is in particular a transition map of the sheaf $I^n_{X'}\subset \cO^n_X$, 
where $I_{X'}\subset \cO_X$ denotes the ideal sheaf of $X'$. Therefore the desired statement is precisely Theorem~\ref{splitting} in the 
special case $\cF=I^n_{X'}$ and $E=X\times \C^n$. 
\end{proof}

In the proof given in~\cite{FP} of the above extension theorem, it has been shown that -- by taking extra-care when dealing with so-called $\cC$-strings -- 
a weakening of the above corollary to the situation where the Cartan pair $(A,B)$ and the subvariety $X'$ of $X$ satisfy 
$X'\cap (A\cap B)\subset (A\cup B)^\circ$ is sufficient for the proof. 
Such extra-care when dealing with $\cC$-strings is no longer needed with the given corollary at hand. Now, one can simply work 
with sufficiently fine $\cC$-covers of $X$, which are well-known to exist. 

\subsection{An unclarified question}
From the viewpoint of Oka theory it is an obvious question if Corollary~\ref{interpolation} allows a similar simplification 
of the proof of extension theorems in settings with Oka (instead of elliptic) fibers. To specify this more carefully 
let us recall the notion of an Oka manifold. A complex manifold $Y$ is called \textit{Oka} if any
holomorphic map from an open neighborhood of a compact convex set $K\subset \C^n$ to $Y$ can be uniformly approximated on 
$K$ by holomorphic maps $\C^n \to Y$. It is known that every elliptic manifold (that is a manifold which admits a dominating spray) 
is an Oka manifold, whereas the converse is open.
The following generalizes Cartan's extension theorem from $Y=\C$ to arbitrary Oka manifolds. 

\begin{theorem}[Forstneri\v c \cite{Interpolation for Oka}]
Let $X$ be a reduced Stein space, $X'$ an analytic subvariety and $Y$ an Oka manifold. 
Then a holomorphic map $f:X' \to Y$ extends holomorphically to $X$ if and only if $f$ extends continuously to $X$.
\end{theorem}

Our impression is that the proof of Forstneri\v c's entension theorem cannot be significantly simplified by Corollary~\ref{interpolation}. 
A new proof using Corollary~\ref{interpolation} seems to depend on an extension of the appropriate 
version of the Oka-Weil theorem for Oka manifold-valued maps -- an important ingredient in the proof -- to a result which includes interpolation on a subvariety.

\appendix

\section{The parametric versions of Theorem~\ref{lifting} and~\ref{splitting}}
Parametric extensions of Theorem~\ref{lifting} and~\ref{splitting} were neglected in the main section since 
their proofs distract from the analytic key difficulties. However, we include them here since 
such extensions are often necessary in the proof of an Oka principle.

\begin{definition}
\label{parametric transition map}
Let $P$ be a compact Hausdorff space. A \textit{parametric transition map} $\gamma_q:U \to E$, $q \in P$, for a given sheaf $\cF$ is a family 
of transition maps (see Definition~\ref{transition map}) which varies continuously in the parameter $q \in P$, that is, 
$P \times U \to E, \ (q,e) \mapsto \gamma_q(e)$ defines a continuous map.
\end{definition}

\begin{remark}
The parametric map $\gamma_q: U \to E$, $q \in P$ is continuous in the sense of Definition~\ref{parametric transition map}
if and only if $q \mapsto \gamma_q$ defines a continuous map from $P$ to the space of continuous maps $U \to E$ equipped with 
the compact open topology. This is well-known and depends only on the fact that $U$ is a locally compact Hausdorff space.
\end{remark}

The only necessary ingredient to extend our results to parametric versions is the following lemma due to Cartan.

\begin{lemma}
\label{Cartans lemma}
Let $\pi: E \to F$ be a surjective linear map of Fr\'echet spaces, let $Q\subset P$ be an inclusion of compact Hausdorff spaces and 
let $f:P \to F$ and $e: Q \to E$ be continuous maps with $\pi \circ e=f|Q$. Then $e$ extends to a continuous map $g:P \to E$ with 
$\pi \circ g=f$.
\end{lemma}

\begin{proof}
See~\cite{Cartan}, Appendice.
\end{proof}

\begin{Theorem 1 with parameters}
\label{parametric lifting}
Let $p:E \to W$ be a vector bundle over a reduced Stein base, let $f_1, \ldots, f_m:W \to E$ be holomorphic sections and set 
$$f:W \times \C^m \to E, \ f(x,z_1, \ldots, z_m)=z_1f_1(x)+\cdots +z_mf_m(x).$$
Moreover let $r>0$ and let $U\subset E$ be an open neighborhood of $f(W \times \B^m_r)$ and let $P$ be a compact Hausdorff space. 
Then, for every parametric transition map $\gamma_q:U \to E$, $q \in P$, of $\cF=\cO_W f_1 + \cdots +\cO_W f_m$, 
there is a parametric transition map $\tilde \gamma_q: W \times \B^m_r \to W \times \C^m$, $q \in P$, with $f \circ \tilde \gamma_q=\gamma_q \circ f|\B^m_r \times W$ for 
every $q \in P$. If $Q\subset P$ is closed and $\gamma_q$ equals the identity for every $q\in Q$, then $\tilde \gamma_q$ can be chosen to be the identity for every $q \in Q$ as well. 
Moreover $\tilde \gamma_q$ can be chosen to approximate the identity as $\gamma_q$ approximates the identity, uniformly in $q \in P$. 
\end{Theorem 1 with parameters}

\begin{proof}
For $D\coloneqq W \times \B^m_r$ our assumptions imply immediately that $q \mapsto \gamma_q \circ f|D$ defines a continuous map $P \to \cF_m(D)$, where $$\cF=\cO_Wf_1+\cdots +\cO_Wf_m.$$ As we reasoned already in the proof of the non-parametric version, $f_1, \ldots, f_m$ 
generate $\cF_m$ as a sheaf of $\cO_{W \times \C^m}$-modules by Lemma~\ref{generators}, and by Lemma~\ref{Theorem B} 
$$\pi: \cO^m(D) \to \cF_m(D), \ (h_1,\ldots, h_m) \mapsto h_1f_1+\cdots +h_mf_m$$ is onto. Since $\cO^m(D) \to \cF_m(D)$ is a surjective linear map of Fr\'echet spaces, 
Cartan's lemma (Lemma~\ref{Cartans lemma}) implies that there is a continuous lift 
$\tilde \gamma: P \to \cO^m(D)$ with $\pi(\tilde \gamma_q)=\gamma_q \circ f|D$ for every $q \in P$, which equals the identity for every fixed $q \in Q$. Such 
$\tilde \gamma:P \to \cO^m(D)$ is the desired lift of the parametric transition map $\gamma$ since $\pi(\tilde \gamma_q)=f \circ \tilde \gamma_q$ by definition.
To see that we can choose the lift $\tilde \gamma_q$ to approximate the identity on $D$ as $\gamma_q$ approximates the identity on $U$, uniformly in $q \in P$, let us note the following: 
the space of continuous maps $P \to \cO^m(D)$, whose values equal the identity for fixed $q \in Q$, and the space of continuous maps $P \to \cF_m(D)$, whose values 
equal the identity for fixed $q \in Q$, are both easily seen the be Fr\'echet spaces if equipped with the compact open topology. 
Moreover, the map between these two spaces obtained by composing elements $P \to \cO^m(D)$ of the former space with $\pi:\cO^m(D) \to \cF_m(D)$ defines a linear map 
of Fr\'echet spaces. The fact that this linear map of Fr\'echet spaces is onto is precisely our statement obtained from Cartan's lemma. Therefore, it follows also in this parametric 
case from the open mapping theorem for Fr\'echet spaces that $\tilde \gamma_q$ can be chosen to approximate the identity as $\gamma_q$ approximates the identity, both approximations 
uniformly in $q \in P$.
This finishes the proof.
\end{proof}

\begin{Theorem 2 with parameters}
\label{parametric splitting}
Let $p:E \to X$ be a vector bundle over a reduced Stein base and identify $X$ with the image of the zero section $X\to E$. Let $(A,B)$ be a Cartan pair in $X$ 
, $U\subset E$ a neighborhood of $A \cap B$ and let 
$\cF$ be a coherent subsheaf of the sheaf of holomorphic sections of $E$. 
Moreover let $Q\subset P$ be an inclusion of compact Hausdorff spaces. 
Then there are neighborhoods $A'\supset A$ and $B'\supset B$ such that for every parametric transition map $\gamma_q: U \to E$, $q \in P$, of $\cF$ which 
is sufficiently close to the identity for every $q \in P$ and which equals the identity for every $q \in Q$, 
there are parametric sections $\alpha:P \to  \cF(A')$ and $\beta:P \to \cF(B')$ such that $\alpha_q=\beta_q=0$ for $q \in Q$ and
$$\gamma_q \circ \alpha_q =\beta_q$$ on $A'\cap B'$ for $q \in P$. 
Moreover, $\alpha_q$ and $\beta_q$ can be chosen to approximate 
the zero sections as $\gamma_q$ approximates the identity, uniformly in $q \in P$. 
\end{Theorem 2 with parameters}

\begin{remark}
\label{parametric special case}
In the special case where $E=X\times \C^m$ and $\cF=\cO^m_X$, the sheaf of sections of $X \times \C^m \to X$, the conclusion of Theorem~\ref{splitting} with parameters is a consequence 
of an analogous parametric version of Proposition 5.8.4, page 238 in~\cite{Francs book}.
\end{remark}

\begin{proof}
We start the same way as in the proof of the non-parametric version:
let $W$ be a Stein neighborhood of $A \cap B$ which is relatively compact in $X\cap U$. 
By Cartan's theorem A there are finitely many global sections 
$f_1, \ldots, f_m \in \cF(X)$ which generate the stalk at every point of $W$, i.e. 
$\cF|W=\cO_{W}f_1+\cdots +\cO_{W}f_m$. Let 
\begin{align*}
f:X \times \C^m \to E, \ f(x,z_1,\ldots, z_m)=z_1f_1(x)+\cdots +z_mf_m(x)
\end{align*}
and $r>0$ such that $f(W \times \B^m_r) \subset U$. 
By Remark~\ref{parametric special case} there are neighborhoods $A'\supset A$ and $B' \supset B$ such that for every parametric transition map $\tilde \gamma_q: W \times \B^m_r \to W \times \C^m$, $q \in P$, which is for every $q \in P$ sufficiently close to the identity and equals the identity for every fixed $q \in Q$, there are $\tilde \alpha:P \to \cO^m_X(A')$ and $\tilde \beta:P  \to \cO^m_X(B')$ with $\tilde \gamma_q \circ \tilde \alpha_q = \tilde \beta_q$ on $A'\cap B'$ and $\tilde \alpha_q, \tilde \beta_q$ equal the zero sections for $q \in Q$. In addition 
$\tilde \alpha_q$ and $\tilde \beta_q$ can be chosen to approximate the zero sections as $\tilde \gamma_q$ approximates the identity, both approximations uniformly in $q \in P$.
It follows from Theorem~\ref{lifting} with parameters that if $\gamma_q:U \to E$, $q \in P$, is a parametric transition map of $\cF$ such that $\gamma_q$ approximates the identity 
sufficiently well uniformly in $q \in P$ and such that $\gamma_q$ equals the identity for fixed $q \in Q$, then there is a lift $\tilde \gamma_q: W \times \B^m_r \to W \times \C^m$, $q \in P$, 
of $\gamma_q:U \to E$ through $f$, which equals the identity for fixed $q \in Q$ and for which we find such a parametric splitting $(\tilde \alpha_q, \tilde \beta_q)$, $q \in P$. Moreover, the splitting $(\tilde \alpha_q, \tilde \beta_q)$ can 
be chosen such that $\alpha_q$ and $\beta_q$ approximate the zero section as $\gamma_q$ (and therefore the chosen lift $\tilde \gamma_q$) approximates the identity, all approximations  
uniformly in $q \in P$. For $\alpha_q= f \circ \tilde \alpha_q \in \cF(A')$ and $\beta_q = f \circ \tilde \beta_q \in \cF(B')$, $q \in P$, we get on $A'\cap B'$ 
\begin{align*}
\gamma_q \circ \alpha_q =\gamma_q \circ f \circ \tilde \alpha_q = f \circ \tilde \gamma_q \circ \tilde \alpha_q =f \circ \tilde \beta_q=\beta_q.
\end{align*}
Therefore $\alpha:P \to \cF(A'), \ q \mapsto \alpha_q$ and $\beta:P \to \cF(B'), \ q \mapsto \beta_q$ yield the desired splitting. This finishes the proof.
\end{proof}


\begin{thebibliography}{99}
\bibitem{Artin} M. Artin: On the solution of analytic equations, Invent. math. 5, 277-291 (1968)
\bibitem{Cartan} H. Cartan: Espaces fibr\'{e}s analytiques, Syposium Internacional de Topologia Algebraica, Mexico (1958) 
\bibitem{Forster und Ramspott} O. Forster, K. J. Ramspott: Okasche Paare von Garben nicht-abelscher Gruppen, Invent. math. 1, 260-286 (1966)
\bibitem{Interpolation for Oka} F. Forstneri\v c: Extending holomorphic mappings from subvarieties in Stein manifolds, Ann. Inst. Fourier 55(3), 1-19 (2005)
\bibitem{Runge implies Oka} F. Forstneri\v c: Runge approximation on convex sets implies the Oka property, Ann. Math. 164(2), 689-707 (2006)
\bibitem{Francs book} F. Forstneri\v{c}: Stein Manifolds and  Holomorphic  Mappings (The Homotopy Principle in Complex Analysis, Second Edition), Ergebnisse der Mathematik und ihrer Grenzgebiete, 3. Folge, Springer-Verlag, Berlin Heidelberg (2017)
\bibitem{FP} F. Forstneri\v c, J. Prezelj: Extending holomorphic sections from closed subvarieties, Math. Z. 236, 43-68 (2001)
\bibitem{elliptic bundles} M. Gromov: Oka's principle for holomorphic sections of elliptic bundles, J. Am. Math. Soc. 2(4), 851-897 (1989)
\bibitem{Gunning/Rossi} R.C. Gunning, H. Rossi: Analytic functions of several complex variables, Prentice-Hall, Inc., Englewood Cliffs, N.J. (1965)
\end{thebibliography}
\end{document}